\documentclass[9pt,oneside]{amsart}
\usepackage[T1]{fontenc}
\usepackage{geometry}
\usepackage{mathtools}
\usepackage{amstext}
\usepackage{amsthm}
\usepackage{amssymb}

\makeatletter
\numberwithin{equation}{section}
\numberwithin{figure}{section}
\theoremstyle{plain}
\newtheorem{thm}{\protect\theoremname}
\theoremstyle{plain}
\newtheorem{cor}[thm]{\protect\corollaryname}
\theoremstyle{plain}
\newtheorem{prop}[thm]{\protect\propositionname}
\theoremstyle{plain}
\newtheorem{lem}[thm]{\protect\lemmaname}
\theoremstyle{definition}
\newtheorem{example}[thm]{\protect\examplename}

\usepackage{listings}
\usepackage{shuffle}

\makeatother

\providecommand{\corollaryname}{Corollary}
\providecommand{\examplename}{Example}
\providecommand{\lemmaname}{Lemma}
\providecommand{\propositionname}{Proposition}
\providecommand{\theoremname}{Theorem}

\begin{document}
\address[Minoru Hirose]{Graduate School of Science and Engineering, Kagoshima University, 1-21-35 Korimoto, Kagoshima, Kagoshima 890-0065, Japan}
\email{hirose@sci.kagoshima-u.ac.jp}

\title{Colored double zeta values and modular forms of general level}
\author{Minoru Hirose}
\begin{abstract}
Gangl, Kaneko, and Zagier gave explicit linear relations among double
zeta values of odd indices coming from the period polynomials of modular
forms for ${\rm SL}(2,\mathbb{Z})$. In this paper, we generalize
their result to the linear relations among colored double zeta values
of level $N$ coming from the modular forms for level $N$ congruence
subgroups.
\end{abstract}

\maketitle

\section{Introduction}

In \cite{GKZ}, Gangl, Kaneko and Zagier gave $\mathbb{Q}$-linear
relations among double zeta values of odd indices $\zeta({\rm odd},{\rm odd})$
by using even period polynomials of modular forms for ${\rm SL}(2,\mathbb{Z})$.
The purpose of this paper is to generalize their result to colored
multiple zeta values and even period polynomials of modular forms
of level $N$. More precisely, Gangl, Kaneko and Zagier's theorem
and our main theorem are formulated as identities in the formal double
zeta space, which is a lift of the space of double zeta values satisfying
the double shuffle relations. First, we will explain concrete examples
of the Gangl-Kaneko-Zagier's result and our main theorem applied to
(colored) double zeta values, and then we will present the precise
formulation using the formal double zeta space.

\subsection{Gangl-Kaneko-Zagier's formula for double zeta values}

Let us start by reviewing the results of Gangl, Kaneko, and Zagier.
We define a multiple zeta value by
\[
\zeta(k_{1},\dots,k_{d}):=\sum_{0<m_{1}<\cdots<m_{d}}\frac{1}{m_{1}^{k_{1}}\cdots m_{d}^{k_{d}}}\in\mathbb{R}
\]
where $k_{1},\dots,k_{d}$ are positive integers with $k_{d}\neq1$.
Then, for even integers $k\geq4$, they considered the $\mathbb{Q}$-linear
relations among $\zeta(r,s)$'s and $\pi^{k}$ where $r$ and $s$
are odd positive integers satisfying $r+s=k$ and $s\geq3$. Then,
numeric calculation suggests that the dimensions of such $\mathbb{Q}$-linear
relations are (conjecturally) equal to $\left\lfloor k/4\right\rfloor -\left\lfloor (k+4)/6\right\rfloor $,
which are the dimension of the space of weight $k$ cusp forms for
${\rm SL}(2,\mathbb{Z})$. For example, when $k=12$, the relations
are (conjecturally) generated by
\[
28\zeta(3,9)+150\zeta(5,7)+168\zeta(5,7)-\frac{5197}{3^{6}5^{3}7^{2}11^{1}13^{1}}\pi^{12}=0.
\]
Furthermore, they construct such linear relations by the period polynomial
of cusp forms. For example, the space of weight $k$ cusp forms for
${\rm SL}(2,\mathbb{Z})$ is generated by the modular discriminant
\[
\Delta(\tau)=q\prod_{n=1}^{\infty}(1-q^{n})^{24}\qquad(q=e^{2\pi i\tau})
\]
and its restricted even period polynomial is a constant multiple of
\[
p_{\Delta}(X,Y)=X^{2}Y^{2}(X^{2}-Y^{2})^{3}.
\]
Then, we have
\begin{equation}
\frac{1}{2880}p_{\Delta}(X-Y,X)=28\frac{X^{2}Y^{8}}{2!8!}-84\frac{X^{3}Y^{7}}{3!7!}+150\frac{X^{4}Y^{6}}{4!6!}-190\frac{X^{5}Y^{5}}{5!5!}+168\frac{X^{6}Y^{4}}{6!4!}-84\frac{X^{7}Y^{3}}{7!3!}.\label{eq:Delta1}
\end{equation}
Then, their result implies the equality
\[
28\zeta(3,9)+150\zeta(5,7)+168\zeta(7,5)=\frac{1}{3}\left(-84\zeta(4,8)-190\zeta(6,6)-84\zeta(8,4)-12\zeta(12)\right)
\]
where all the coefficients come from the ones in \ref{eq:Delta1}
(the coefficients $-12$ of $\zeta(12)$ comes from the sum of all
coefficients $28-84+150-190+168-84$). Here, since each $\zeta(2a,2b)+\zeta(2b,2a)$
and $\zeta(\mathrm{even})$ are rational multiple of power of $\pi$,
this shows that $28\zeta(3,9)+150\zeta(5,7)+168\zeta(7,5)$ is a rational
multiple of $\pi^{12}$.

\subsection{An example of main theorem: modular phenomena for certain Dirichlet
series and the cusp form for $X_{0}(11)$\label{subsec:intro_X0_11}}

Now, let us see one example of the main result of this paper. For
a rational prime $p\geq3$ and $a\in\mathbb{Z}/p\mathbb{Z}$, put
\[
L_{p}(a):=\sum_{0<m<n}\frac{\chi_{p}(am+n)}{mn}
\]
and
\[
L_{p}^{\pm}(a):=L_{p}(a)\pm L_{p}(-a)
\]
where
\[
\chi_{p}(t):=\begin{cases}
p-1 & t\in p\mathbb{Z}\\
-1 & t\notin p\mathbb{Z}.
\end{cases}
\]
Then, for a fixed $p$, let us consider the $\mathbb{Q}$-linear relations
among $L_{p}^{+}(a)$'s for $1\leq a\leq\frac{p-1}{2}$ and $\pi^{2}$.
Then, numeric calculation suggests that the dimensions of such $\mathbb{Q}$-linear
relations are (conjecturally) given by $\left\lfloor (p+1)/4\right\rfloor -\left\lfloor (p+5)/6\right\rfloor $,
which are surprisingly equal to the dimension of weight $2$ cusp
forms for $\Gamma_{0}(p)$ (\cite[Exercise 3.1.4]{DS_modular}). For
example, for $p=11$, the relations are (conjecturally) generated
by
\begin{equation}
-2L_{11}^{+}(1)+3L_{11}^{+}(2)+L_{11}^{+}(3)-L_{11}^{+}(4)-L_{11}^{+}(5)=\frac{71\pi^{2}}{33}.\label{eq:exa_Lsum_11}
\end{equation}
Furthermore, in this paper, we construct such linear relations by
the period polynomial of cusp forms as follows. For example, the space
of weight $2$ cusp forms for $\Gamma_{0}(11)$ is generated by
\[
f(\tau)=q\prod_{n=1}^{\infty}(1-q^{n})(1-q^{11n})\qquad(q=e^{2\pi i\tau}).
\]
Define the map $Q_{f},Q_{f}^{\mathrm{od}},Q_{f}^{\mathrm{ev}}:\{1,\dots,10\}\to\mathbb{C}$
by
\[
Q_{f}(a)=\begin{cases}
\int_{\frac{1}{a-1}}^{\frac{1}{a}}f(z)dz+\int_{\frac{-1}{a-1}}^{\frac{-1}{a}}f(z)dz & a\neq1\\
0 & a=1
\end{cases}
\]
and
\[
Q_{f}^{\mathrm{od}}(a)=\frac{Q_{f}(a)+Q_{f}(11-a)}{2},\ Q_{f}^{\mathrm{ev}}(a)=\frac{Q_{f}(a)-Q_{f}(11-a)}{2}.
\]
Then,
\begin{align}
(Q_{f}(1),\dots,Q_{f}(10)) & =c_{f}\cdot(0,4,-2,-4,-2,0,2,4,2,-4)\nonumber \\
(Q_{f}^{\mathrm{od}}(1),\dots,Q_{f}^{\mathrm{od}}(10)) & =c_{f}\cdot(-2,3,1,-1,-1,-1,-1,1,3,-2)\label{eq:Qfodd}\\
(Q_{f}^{\mathrm{ev}}(1),\dots,Q_{f}^{\mathrm{ev}}(10)) & =c_{f}\cdot(2,1,-3,-3,-1,1,3,3,-1,-2)\label{eq:Qfeven}
\end{align}
where $c_{f}\in\mathbb{C}^{\times}$ is a certain constant. Then main
result of this paper implies
\begin{align}
 & 3\left(-2L_{11}(1)+3L_{11}(2)+L_{11}(3)-L_{11}(4)-L_{11}(5)-L_{11}(6)-L_{11}(7)+L_{11}(8)+3L_{11}(9)-2L_{11}(10)\right)\nonumber \\
 & =-\left(2L_{11}(1)+L_{11}(2)-3L_{11}(3)-3L_{11}(4)-L_{11}(5)+L_{11}(6)+3L_{11}(7)+3L_{11}(8)-L_{11}(9)-2L_{11}(10)\right)\label{eq:exa_Lp_eq_gamma_11}\\
 & \qquad-Q_{f}(10)\left(11-\frac{1}{11}\right)\zeta(2),\nonumber 
\end{align}
where the coefficients come from (\ref{eq:Qfodd}) and (\ref{eq:Qfeven}).
Since
\[
L_{p}(c)+L_{p}(c^{-1})-L_{p}(-c)-L_{p}(c^{-1})=-2s(c,p)\pi^{2}+(\delta_{c,1}-\delta_{c,-1})(p-1/p)\frac{\pi^{2}}{6}
\]
for $c\in(\mathbb{Z}/p\mathbb{Z})^{\times}$ where $s(c,p)\in\mathbb{Q}$
is a Dedekind sum, (\ref{eq:exa_Lp_eq_gamma_11}) implies (\ref{eq:exa_Lsum_11}).

\subsection{Period polynomials for congruence subgroups}

Let us recall the theory of period polynomials for congruence subgroups
in \cite{Pasol_Popa_periodpolynomial}. Fix a nonnegative integer
$w$. We denote by $V_{w}$ the space of degree $w$ homogeneous polynomials
in $X$ and $Y$ with rational coefficients. We define the right action
of ${\rm GL}(2,\mathbb{Z})$ on $V_{w}$ by
\[
\left.P(X,Y)\right|_{\gamma}=P(aX+bY,cX+dY)\quad\text{for }\gamma=\left(\begin{array}{cc}
a & b\\
c & d
\end{array}\right)\in{\rm GL}(2,\mathbb{Z}).
\]
We use the following notations for some elements of ${\rm GL}(2,\mathbb{Z})$:
\[
\epsilon\coloneqq\left(\begin{array}{cc}
-1 & 0\\
0 & 1
\end{array}\right),\ J\coloneqq\left(\begin{array}{cc}
-1 & 0\\
0 & -1
\end{array}\right),\ S\coloneqq\left(\begin{array}{cc}
0 & -1\\
1 & 0
\end{array}\right),\ U\coloneqq\left(\begin{array}{cc}
1 & -1\\
1 & 0
\end{array}\right),\ T\coloneqq US^{-1}=\left(\begin{array}{cc}
1 & 1\\
0 & 1
\end{array}\right).
\]
Let $\Gamma$ be a congruence subgroup of ${\rm SL}(2,\mathbb{Z})$
such that $\epsilon\Gamma\epsilon=\Gamma$. We denote by $\tilde{V}_{w}^{\Gamma}$
the space of maps $P:\Gamma\backslash{\rm SL}(2,\mathbb{Z})\to V_{w}$.
We define the right action of ${\rm GL}(2,\mathbb{Z})$ on $\tilde{V}_{w}^{\Gamma}$
by $\left.P\right|_{\gamma}(C)=\left.P(C\gamma^{-1})\right|_{\gamma}$
for $\gamma\in{\rm SL}(2,\mathbb{Z})$ and $\left.P\right|_{\epsilon}(C)=\left.P(\epsilon C\epsilon)\right|_{\epsilon}$.
We define the subspaces $\tilde{V}_{w}^{\Gamma}\supset V_{w}^{\Gamma}\supset W_{w}^{\Gamma}\supset C_{w}^{\Gamma}$
by
\[
V_{w}^{\Gamma}\coloneqq\{P\in\tilde{V}_{w}^{\Gamma}:\left.P\right|_{J}=P\},
\]
\[
W_{w}^{\Gamma}\coloneqq\{P\in V_{w}^{\Gamma}:\left.P\right|_{1+S}=\left.P\right|_{1+U+U^{2}}=0\},
\]
and
\[
C_{w}^{\Gamma}\coloneqq\{\left.P\right|_{1-S}\,:\,P\in V_{w}^{\Gamma},\ \left.P\right|_{T}=P\}.
\]
Since the action by $\epsilon$ preserve $W_{w}^{\Gamma}$ and $C_{w}^{\Gamma}$,
they are decomposed into $\pm1$-eigenspaces, denoted by $W_{w}^{\Gamma,\pm}$
and $C_{w}^{\Gamma,\pm}$, respectively. For $f\in S_{w+2}(\Gamma)$
and $C\in\Gamma\backslash{\rm SL}(2,\mathbb{Z})$, define $\left.f\right|_{C}\in S_{w+2}(C^{-1}\Gamma C)$
by $\left.f\right|_{C}(z)=(cz+d)^{-w-2}f(\frac{az+b}{cz+d})$ where
$\left(\begin{smallmatrix}a & b\\
c & d
\end{smallmatrix}\right)\in C$. Furthermore, for $f\in S_{w+2}(\Gamma)$, define $\rho_{f}\in V_{w}^{\Gamma}\otimes\mathbb{C}$
by 
\[
\rho_{f}(C)=\int_{0}^{i\infty}\left.f\right|_{C}(z)\cdot(zY-X)^{w}dz.
\]
Then we can show $\rho_{f}\in W_{w}^{\Gamma}\otimes\mathbb{C}$ for
$f\in S_{w+2}(\Gamma)$. We put $\rho_{f}^{\pm}=\frac{1}{2}(\rho_{f}\pm\left.\rho_{f}\right|_{\epsilon})$.
Then, Popa-Pasol's restatement \cite[Theorem 2.1]{Pasol_Popa_periodpolynomial}
of Eichler--Shimura isomorphism says that the maps $S_{w+2}(\Gamma)\to W_{w}^{\Gamma,\pm}\otimes\mathbb{C}$,
$f\mapsto\rho_{f}^{\pm}$ give rise to isomorphisms
\[
S_{w+2}(\Gamma)\simeq(W_{w}^{\Gamma,\pm}/C_{w}^{\Gamma,\pm})\otimes\mathbb{C}.
\]
Furthermore, we put
\[
\bar{W}_{w}^{\Gamma,\pm}\coloneqq\begin{cases}
W_{w}^{\Gamma,\pm} & w:\,{\rm even}\\
W_{w}^{\Gamma,\mp} & w:\,{\rm odd}.
\end{cases}
\]
Note that $W_{w}^{{\rm SL}(2,\mathbb{Z}),\pm}$ is canonically isomorphic
to 
\[
W_{w}^{\pm}\coloneqq\{P\in V_{w}:\left.P\right|_{1+S}=\left.P\right|_{1+U+U^{2}}=0,\,\left.P\right|_{\epsilon}=\pm P\}
\]
treated in \cite{GKZ}.

\subsection{Formal double zeta space of level $N$}

Let us return to the double zeta values. The colored multiple zeta
values of level $N$ are generalization of multiple zeta values defined
by
\[
\zeta{k_{1},\dots,k_{d} \choose z_{1},\dots,z_{d}}=\sum_{0<m_{1}<\cdots<m_{d}}\frac{z_{1}^{m_{1}}\cdots z_{d}^{m_{d}}}{m_{1}^{k_{1}}\cdots m_{d}^{k_{d}}}\in\mathbb{C}
\]
where $k_{1},\dots,k_{d}$ are positive integers and $z_{1},\dots,z_{d}$
are $N$-th roots of unity with $(k_{d},z_{d})\neq(1,1)$. Furthermore,
we define the shuffle (resp. harmonic) regularized colored multiple
zeta values $\zeta^{\shuffle}{k_{1},\dots,k_{d} \choose z_{1},\dots,z_{d}}\in\mathbb{C}[T]$
(resp. $\zeta^{*}{k_{1},\dots,k_{d} \choose z_{1},\dots,z_{d}}\in\mathbb{C}[T]$)
by shuffle (resp. harmonic) regularization. They equal to $\zeta{k_{1},\dots,k_{d} \choose z_{1},\dots,z_{d}}$
if $(k_{d},z_{d})\neq(1,1)$ and satisfy the regularized double shuffle
relations (\cite{Rac02}, \cite{AK04}, \cite[Chapter 13]{ZhaoBook}):
\begin{align}
\zeta^{\shuffle}\binom{r}{\zeta_{N}^{a}}\zeta^{\shuffle}\binom{s}{\zeta_{N}^{b}} & =\sum_{j=0}^{r-1}\binom{s-1+j}{j}\zeta^{\shuffle}\binom{r-j,s+j}{\zeta_{N}^{a-b},\zeta_{N}^{b}}+\sum_{j=0}^{s-1}{r-1+j \choose j}\zeta^{\shuffle}{s-j,r+j \choose \zeta_{N}^{b-a},\zeta_{N}^{a}}\nonumber \\
 & =\zeta^{*}\binom{r,s}{\zeta_{N}^{a},\zeta_{N}^{b}}+\zeta^{*}\binom{s,r}{\zeta_{N}^{b},\zeta_{N}^{a}}+\zeta\binom{r+s}{\zeta_{N}^{a+b}}\label{eq:dsh_intro}
\end{align}
where $\zeta_{N}\coloneqq\exp(2\pi i/N)$. Furthermore, it is known
that $\zeta^{*}\binom{r,s}{\zeta_{N}^{a},\zeta_{N}^{b}}=\zeta^{\shuffle}\binom{r,s}{\zeta_{N}^{a},\zeta_{N}^{b}}$
except for the case $r=s=\zeta_{N}^{a}=\zeta_{N}^{b}=1$. Let $k\geq2$
and $N\geq1$ with $(k,N)\neq(2,1)$. Based on the regularized double
shuffle relations (\ref{eq:dsh_intro}), we define the \emph{formal
double zeta space} of level $N$ and weight $k$ as the $\mathbb{Q}$-vector
space $\mathcal{D}_{k,N}$ with generators 
\[
Z_{a,b}^{r,s},P_{a,b}^{r,s},Z_{c}^{k}\quad(r+s=k,\,r,s\geq1,\,a,b,c\in\mathbb{Z}/N\mathbb{Z},\,{\rm gcd}(a,b,N)=1)
\]
and relations
\[
P_{a,b}^{r,s}=\sum_{j=0}^{r-1}\binom{s-1+j}{j}Z_{a-b,b}^{r-j,s+j}+\sum_{j=0}^{s-1}{r-1+j \choose j}Z_{b-a,a}^{s-j,r+j}=Z_{a,b}^{r,s}+Z_{b,a}^{s,r}+Z_{a+b}^{r+s}.
\]
Furthermore, based on the property
\begin{equation}
\zeta\binom{k}{\zeta_{N}^{a}}+(-1)^{k}\zeta\binom{k}{\zeta_{N}^{-a}}=\sum_{m\neq0}\frac{\zeta_{N}^{am}}{m^{k}}=-\frac{B_{k}(a/N)}{k!}(2\pi i)^{k}\in\mathbb{Q}\cdot(2\pi i)^{k},\label{eq:euler_intro}
\end{equation}
define $\mathcal{P}_{k,N}^{{\rm ev}}\subset\mathcal{D}_{k,N}$ as
the $\mathbb{Q}$-vector subspace spanned by
\[
\left\{ P_{a,b}^{r,s}+(-1)^{r}P_{-a,b}^{r,s}+(-1)^{s}P_{a,-b}^{r,s}+(-1)^{r+s}P_{-a,-b}^{r,s}:\substack{r+s=k,r,s\geq1\\
a,b\in\mathbb{Z}/N\mathbb{Z},\,\gcd(a,b,N)\equiv1
}
\right\} \cup\left\{ Z_{a}^{k}+(-1)^{k}Z_{-a}^{k}:a\in\mathbb{Z}/N\mathbb{Z}\right\} .
\]
This is a generalization of the formal double zeta space $\mathcal{D}_{k}$
and its subspace $\mathcal{P}_{k}^{{\rm ev}}$ introduced in \cite{GKZ}.
Note that by (\ref{eq:dsh_intro}) and (\ref{eq:euler_intro}) the
$\mathbb{Q}$-linear map
\[
\Phi_{k,N}:\mathcal{D}_{k,N}\to\mathbb{C}\ ;\ Z_{a,b}^{r,s}\mapsto\zeta^{\shuffle}\binom{r,s}{\zeta_{N}^{a},\zeta_{N}^{b}},\,P_{a,b}^{r,s}\mapsto\zeta^{\shuffle}\binom{r}{\zeta_{N}^{a}}\zeta^{\shuffle}\binom{s}{\zeta_{N}^{b}},\,Z_{c}^{k}\mapsto\zeta^{\shuffle}\binom{k}{\zeta_{N}^{c}}
\]
is well-defined and
\[
\Phi_{k,N}(\mathcal{P}_{k,N}^{{\rm ev}})\subset\mathbb{Q}\cdot(2\pi i)^{k}.
\]
Note that the $\mathbb{Q}$-linear combination $\sum_{r,s,a,b}c_{a,b}^{r,s}Z_{a,b}^{r,s}$
is in $\Phi(\mathcal{P}_{k,N}^{{\rm ev}})$ if the coefficients have
the properties
\[
c_{a,b}^{r,s}=c_{b,a}^{s,r}=(-1)^{r}c_{-a,b}^{r,s}=(-1)^{s}c_{a,-b}^{r,s}
\]
for all $r,s,a,b$. Furthermore, as an analogy of $\zeta^{\mathfrak{m}}({\rm odd},{\rm odd})$,
define the set of (odd,odd)-type formal colored double zetas of level
$N$ as
\[
\left\{ \frac{1}{4}\sum_{\rho,\sigma\in\{\pm1\}}\rho^{r+1}\sigma^{s+1}Z_{\rho a,\sigma b}^{r,s}\,:\,\begin{aligned}0\leq a,b\leq N/2,\\
\text{\ensuremath{2a\notin N\mathbb{Z}} or \ensuremath{r} is odd},\\
\text{\ensuremath{2b\notin N\mathbb{Z}} or \ensuremath{s} is odd}
\end{aligned}
\right\} .
\]
Note that $\mathbb{Q}$-linear combinations of (odd,odd)-type formal
colored double zetas are equivalent to the sums of the form
\[
\sum_{a,b,r,s}c_{a,b}^{r,s}Z_{a,b}^{r,s}
\]
with
\[
c_{a,b}^{r,s}=(-1)^{r+1}c_{-a,b}^{r,s}=(-1)^{s+1}c_{a,-b}^{r,s}.
\]
Under this setting, Gangl--Kaneko--Zagier's result can be stated
as follows.
\begin{thm}[\cite{GKZ}]
\label{thm:GKZ}Fix an even positive integer $k$ and put $w=k-2$.
Let $P(X,Y)\in W_{w}^{+}$ and define $q^{r,s}\in\mathbb{Q}$ by
\[
P(X-Y,X)=\sum_{\substack{r+s=w\\
r,s\geq0
}
}\frac{1}{r!s!}q^{r+1,s+1}X^{r}Y^{s}.
\]
Put $Z^{r,s}\coloneqq Z_{0,0}^{r,s}\in\mathcal{D}_{r+s,1}$ and $Z^{k}\coloneqq Z_{0}^{k}\in\mathcal{D}_{k,1}$.
Then $q_{r,s}=q_{s,r}$ for $r,s$ even and
\[
3\sum_{\substack{r+s=k\\
r,s:{\rm odd}
}
}q^{r,s}Z^{r,s}=-\sum_{\substack{r+s=k\\
r,s,{\rm even}
}
}q^{r,s}Z^{r,s}-(\sum_{r+s=k}q^{r,s})Z^{k}\quad\in\mathcal{P}_{k,1}^{{\rm ev}}.
\]
Conversely, all the linear combinations of $Z^{r,s}$ with $r,s\geq1$
odd belonging to $\mathcal{P}_{k,1}^{{\rm ev}}$ are obtained by this
way.
\end{thm}

\begin{cor}[\cite{GKZ}]
\label{cor:GKZ}Let $P(X,Y)\in W_{w}^{+}$ and $q^{r,s}\in\mathbb{Q}$
be as in Theorem \ref{thm:GKZ}. Then
\[
\sum_{\substack{r+s=k\\
r,s:{\rm odd}
}
}q^{r,s}\zeta(r,s)\in\mathbb{Q}\cdot(2\pi i)^{k}.
\]
\end{cor}

Let $A(N)$ be the set of pairs $(a,b)\in\mathbb{Z}/N\mathbb{Z}$
satisfying ${\rm gcd}(a,b,N)=1$. Then there are bijections
\[
\Gamma_{1}(N)\backslash{\rm SL}(2,\mathbb{Z})\simeq A(N)\quad;\quad\Gamma_{1}(N)\left(\begin{array}{cc}
a & b\\
c & d
\end{array}\right)\mapsto(c,d),
\]
and
\[
\Gamma_{0}(N)\backslash{\rm SL}(2,\mathbb{Z})\simeq\mathbb{P}^{1}(\mathbb{Z}/N\mathbb{Z})\quad;\quad\Gamma_{0}(N)\left(\begin{array}{cc}
a & b\\
c & d
\end{array}\right)\mapsto(c:d).
\]
For $(a,b)\in A(N)$ (resp. $\alpha\in\mathbb{P}^{1}(\mathbb{Z}/N\mathbb{Z})$),
we denote the corresponding element in $\Gamma_{1}(N)\backslash{\rm SL}(2,\mathbb{Z})$
(resp. $\Gamma_{0}(N)\backslash{\rm SL}(2,\mathbb{Z})$) by $C_{a,b}$
(resp. $C_{\alpha}$). The followings are $\Gamma=\Gamma_{1}(N)$
and $\Gamma=\Gamma_{0}(N)$ cases of the main theorem (Theorem \ref{thm:main})
of this paper, which give generalizations of Theorem \ref{thm:GKZ}.
\begin{thm}
\label{thm:main_intro}Fix $N\geq1$ and $w\geq0$, and put $k=w+2$.
Let $P\in\bar{W}_{w}^{\Gamma_{1}(N),+}$ and define $q_{a,b}^{r,s}\in\mathbb{Q}$
for $r,s,a,b$ with $r,s\geq1$, $r+s=k$, and $(a,b)\in A(N)$ by
\[
P(C_{a,-a+b})(X-Y,X)=\sum_{r,s}\frac{1}{r!s!}q_{a,b}^{r+1,s+1}X^{r}Y^{s},
\]
and put 
\[
q_{a,b}^{r,s,{\rm ev}}=\frac{1}{2}\left(q_{a,b}^{r,s}+(-1)^{r}q_{-a,b}^{r,s}\right),\quad q_{a,b}^{r,s,{\rm od}}=\frac{1}{2}\left(q_{a,b}^{r,s}-(-1)^{r}q_{-a,b}^{r,s}\right).
\]
Then
\[
q_{a,b}^{r,s,{\rm od}}=(-1)^{r+1}q_{-a,b}^{r,s,{\rm od}}=(-1)^{s+1}q_{a,-b}^{r,s,{\rm od}},\quad q_{a,b}^{r,s,{\rm ev}}=(-1)^{r}q_{-a,b}^{r,s,{\rm ev}}=(-1)^{s}q_{a,-b}^{r,s,{\rm ev}}=q_{b,a}^{s,r,{\rm ev}}
\]
and
\begin{align*}
3\sum_{r,s,a,b}q_{a,b}^{r,s,{\rm od}}Z_{a,b}^{r,s} & =-\sum_{r,s,a,b}q_{a,b}^{r,s,{\rm ev}}Z_{a,b}^{r,s}-\sum_{r,s,a,b}q_{a,b}^{r,s}Z_{a+b}^{r+s}\quad\in\mathcal{P}_{k,N}^{{\rm ev}}.
\end{align*}
Conversely, a linear combination
\[
\sum_{r,s,a,b}c_{a,b}^{r,s}Z_{a,b}^{r,s}\quad(c_{a,b}^{r,s}=(-1)^{r+1}c_{-a,b}^{r,s}=(-1)^{s+1}c_{a,-b}^{r,s})
\]
belongs to $\mathcal{P}_{k,N}^{{\rm ev}}$ if and only if $c_{a,b}^{r,s}=q_{a,b}^{r,s,{\rm od}}$
arising in this way.
\end{thm}

\begin{thm}
\label{thm:main_intro_gamma0}Fix $N\geq1$ and $w\geq0$, and put
$k=w+2$. Assume that $w$ is even. For $\alpha\in\mathbb{P}^{1}(\mathbb{Z}/N\mathbb{Z})$
and positive integers $r,s$ with $r+s=k$, define $Z_{\alpha}^{r,s}\in\mathcal{D}_{k,N}$
and $Z_{\alpha}^{k}$ by
\[
Z_{\alpha}^{r,s}=\sum_{\substack{a,b\in\mathbb{Z}/N\mathbb{Z}\\
(a:b)=\alpha
}
}Z_{a,b}^{r,s},\qquad Z_{\alpha}^{k}=\sum_{\substack{a,b\in\mathbb{Z}/N\mathbb{Z}\\
(a:b)=\alpha
}
}Z_{a+b}^{k}.
\]
Fix $P\in\bar{W}_{w}^{\Gamma_{1}(N),+}$. Define rational numbers
\[
q_{\alpha}^{r,s}\in\mathbb{Q}\qquad\qquad(r,s\geq1,\ r+s=k,\ \alpha\in\mathbb{P}^{1}(\mathbb{Z}/N\mathbb{Z}))
\]
by 
\[
P(C_{(a:-a+b)})(X-Y,X)=\sum_{r,s}\frac{1}{r!s!}q_{(a:b)}^{r+1,s+1}X^{r}Y^{s},
\]
and put 
\[
q_{\alpha}^{r,s,{\rm ev}}=\frac{1}{2}\left(q_{\alpha}^{r,s}+(-1)^{r}q_{-\alpha}^{r,s}\right),\quad q_{\alpha}^{r,s,{\rm od}}=\frac{1}{2}\left(q_{\alpha}^{r,s}-(-1)^{r}q_{-\alpha}^{r,s}\right).
\]
Then
\[
q_{\alpha}^{r,s,{\rm od}}=(-1)^{r+1}q_{-\alpha}^{r,s,{\rm od}},\quad q_{\alpha}^{r,s,{\rm ev}}=(-1)^{r}q_{-\alpha}^{r,s,{\rm ev}},\quad q_{(a:b)}^{r,s,{\rm ev}}=q_{(b:a)}^{s,r,{\rm ev}}
\]
and
\begin{align*}
3\sum_{r,s,\alpha}q_{\alpha}^{r,s,{\rm od}}Z_{\alpha}^{r,s} & =-\sum_{r,s,\alpha}q_{\alpha}^{r,s,{\rm ev}}Z_{\alpha}^{r,s}-\sum_{r,s,\alpha}q_{\alpha}^{r,s}Z_{\alpha}^{r+s}\quad\in\mathcal{P}_{k,N}^{{\rm ev}}.
\end{align*}
Conversely, a linear combination
\[
\sum_{r,s,\alpha}c_{\alpha}^{r,s}Z_{\alpha}^{r,s}\quad(c_{\alpha}^{r,s}=(-1)^{r+1}c_{-\alpha}^{r,s})
\]
belongs to $\in\mathcal{P}_{k,N}^{{\rm ev}}$ if and only if $c_{\alpha}^{r,s}=q_{\alpha}^{r,s,{\rm od}}$
arising in this way.
\end{thm}

As corollaries, we have the following generalizations of Corollary
\ref{cor:GKZ}:
\begin{cor}
\label{cor:intro_main}Fix $N\geq1$ and $w\geq0$ with $(N,w)\neq(1,0)$.
Put $k=w+2$. Let $P\in\bar{W}_{w}^{\Gamma_{1}(N),+}$ and $q_{a,b}^{r,s,{\rm od}}\in\mathbb{Q}$
be as in Theorem \ref{thm:main_intro}. Then
\[
\sum_{r,s,a,b}q_{a,b}^{r,s,{\rm od}}\zeta^{\shuffle}{r,s \choose \zeta_{N}^{a},\zeta_{N}^{b}}\in\mathbb{Q}\cdot(2\pi i)^{k}.
\]
\end{cor}

\begin{cor}
\label{cor:intro_main_gamma0}Fix $N\geq1$ and even $w\geq0$ with
$(N,w)\neq(1,0)$. Put $k=w+2$. Let $P\in\bar{W}_{w}^{\Gamma_{0}(N),+}$
and $q_{\alpha}^{r,s,{\rm od}}\in\mathbb{Q}$ be as in Theorem \ref{thm:main_intro}.
Then
\[
\sum_{r,s,\alpha}q_{\alpha}^{r,s,{\rm od}}L{r,s \choose \alpha}\in\mathbb{Q}\cdot(2\pi i)^{k}
\]
where
\[
L{r,s \choose \alpha}:=\sum_{\substack{a,b\in\mathbb{Z}/N\mathbb{Z}\\
(a:b)=\alpha
}
}\zeta^{\shuffle}{r,s \choose \zeta_{N}^{a},\zeta_{N}^{b}}.
\]
\end{cor}

Note that $L{r,s \choose \alpha}$ coincides with $L(b)$ if $N$
is a prime and $\alpha=(1:b)$, and thus, Corollary \ref{cor:intro_main_gamma0}
is also a generalization of (\ref{eq:exa_Lsum_11}).

\section{Proof of main theorem}

In this section, we give a proof of Theorem \ref{thm:main_intro}.

\subsection{The dual vector space of $V_{w}^{\Gamma}$}

Fix an integer $w\ge0$ and a congruence subgroup $\Gamma$ of ${\rm SL}(2,\mathbb{Z})$
such that $\epsilon\Gamma\epsilon=\Gamma$. We denote by $\tilde{\bar{V}}_{w}^{\Gamma}$
(resp. $\bar{V}_{w}^{\Gamma}$) a ${\rm GL}(2,\mathbb{Z})$-module
which is just a copy of $\tilde{V}_{w}^{\Gamma}$ (resp. $V_{w}^{\Gamma}$)
as a ${\rm SL}(2,\mathbb{Z})$-module, but the action of $\epsilon$
is defined by
\[
\left.P\right|_{\epsilon}(C)=\left.P(J\epsilon C\epsilon)\right|_{\epsilon}
\]
for $C\in\Gamma\backslash{\rm SL}(2,\mathbb{Z})$. Note that $\bar{W}_{w}^{\Gamma,\pm}$
is naturally identified with
\[
\{P\in\bar{V}_{w}^{\Gamma}:\left.P\right|_{1+S}=\left.P\right|_{1+U+U^{2}}=0,\,\left.P\right|_{\epsilon}=\pm P\}.
\]
Following \cite{Pasol_Popa_periodpolynomial}, define pairings on
$V_{w}\times V_{w}$ and $V_{w}^{\Gamma}\times\bar{V}_{w}^{\Gamma}$
by 
\[
\langle\sum_{r=0}^{w}a_{r}X^{r}Y^{w-r},\sum_{r=0}^{w}b_{r}X^{r}Y^{w-r}\rangle=\sum_{r=0}^{w}(-1)^{w-r}{w \choose r}^{-1}a_{w}b_{w-r}
\]
and
\[
\langle\langle P,Q\rangle\rangle=\frac{1}{[{\rm SL}(2,\mathbb{Z}):\Gamma]}\sum_{C\in\Gamma\backslash{\rm SL}(2,\mathbb{Z})}\langle P(C),Q(C)\rangle\qquad(P\in V_{w}^{\Gamma},Q\in\bar{V}_{w}^{\Gamma}).
\]
By this pairing, we regard $\bar{V}_{w}^{\Gamma}$ as the dual vector
space of $V_{w}^{\Gamma}$. Then for $P,Q\in V_{w}$ and $g\in{\rm GL}(2,\mathbb{Z})$,
we have $\langle\left.P\right|_{g},Q\rangle=\langle P,\left.Q\right|_{g^{\vee}}\rangle$
where $g^{\vee}=\det(g)g^{-1}\in{\rm GL}(2,\mathbb{Z})$ (see \cite[Section 3]{Pasol_Popa_periodpolynomial}).
Thus, we have
\[
\langle\langle\left.P\right|_{g},Q\rangle\rangle=\langle\langle P,\left.Q\right|_{g^{-1}}\rangle\rangle
\]
for $P\in V_{w}^{\Gamma}$, $Q\in\bar{V}_{w}^{\Gamma}$, and $g\in{\rm GL}(2,\mathbb{Z})$.

\subsection{Properties of finite dimensional vector spaces with $\mathrm{PGL}(2,\mathbb{Z})$-action}

Let $A$ be a finite dimensional $\mathbb{Q}$-vector space with a
right $\mathrm{PGL}(2,\mathbb{Z})$-action. For such $A$, we define
the subspaces $\mathcal{W}(A)$, $A^{\pm}$ and $\mathcal{W}^{\pm}(A)$
of $A$ by
\[
\mathcal{W}(A)=\{P\in A:\left.P\right|_{1+S}=\left.P\right|_{(1+U+U^{2})(1-S)}=0\},
\]
\[
A^{\pm}=\{P\in A:\left.P\right|_{\epsilon}=\pm P\},
\]
\[
\mathcal{W}^{\pm}(A)=\mathcal{W}(A)\cap A^{\pm}.
\]
Note that $\mathcal{W}(A)=\mathcal{W}^{+}(A)\oplus\mathcal{W}^{-}(A)$
since $\epsilon$ preserves $\mathcal{W}(A)$. Furthermore, we denote
by $A^{\vee}$ the dual vector space of $A$. We also regard $A^{\vee}$
as a right $\mathrm{PGL}(2,\mathbb{Z})$-module by the action
\begin{equation}
\langle\left.P\right|_{g},Q\rangle=\langle P,\left.Q\right|_{g^{-1}}\rangle\qquad(P\in A^{\vee},Q\in A,g\in\mathrm{PGL}(2,\mathbb{Z})).\label{eq:dual_action}
\end{equation}
When $A=V_{w}^{\Gamma}$ or $\bar{V}_{w}^{\Gamma}$, $\mathcal{W}^{\pm}(A)$
coincides with the set of period polynomials.
\begin{prop}
\label{prop:W_general}We have
\[
\mathcal{W}^{\pm}(V_{w}^{\Gamma})=W_{w}^{\Gamma,\pm},
\]
and
\[
\mathcal{W}^{\pm}(\bar{V}_{w}^{\Gamma})=\bar{W}_{w}^{\Gamma,\pm}.
\]
\end{prop}

\begin{proof}
Note that it is enough to show that $\mathcal{W}(V_{w}^{\Gamma})=W_{w}^{\Gamma}$.
Recall that the definitions of $\mathcal{W}(V_{w}^{\Gamma})$ and
$W_{w}^{\Gamma}$ are given by
\begin{align*}
\mathcal{W}(V_{w}^{\Gamma}) & =\{P\in V_{w}^{\Gamma}:\left.P\right|_{1+S}=0,\,\left.P\right|_{(1+U+U^{2})(1-S)}=0\},\\
W_{w}^{\Gamma} & =\{P\in V_{w}^{\Gamma}:\left.P\right|_{1+S}=0,\,\left.P\right|_{1+U+U^{2}}=0\}.
\end{align*}
Thus the proposition is equivalent to $\left.P\right|_{1+U+U^{2}}=0$
for $P\in\mathcal{W}(V_{w}^{\Gamma})$. Let $P$ be any element of
$\mathcal{W}(V_{w}^{\Gamma})$ and put $Q=\left.P\right|_{1+U+U^{2}}$.
Then
\[
\left.Q\right|_{S}=\left.Q\right|_{U}=Q.
\]
Since ${\rm SL}(2,\mathbb{Z})$ is generated by $S$ and $U$, we
have $Q=\left.Q\right|_{g}$ for $g\in{\rm SL}(2,\mathbb{Z})$. Since
$\Gamma$ is a finite index subgroup of ${\rm SL}(2,\mathbb{Z})$,
there exists a positive integer $n$ such that $CT^{n}=C$ for any
$C\in\Gamma\backslash{\rm SL}(2,\mathbb{Z})$. Then, for $C\in\Gamma\backslash{\rm SL}(2,\mathbb{Z})$,
we have
\[
Q(C)=\left.Q\right|_{T^{n}}(C)=\left.Q(CT^{-n})\right|_{T^{n}}=\left.Q(C)\right|_{T^{n}},
\]
which implies 
\[
Q(C)\in\mathbb{Q}Y^{w}.
\]
Furthermore, we also have
\[
Q(C)=\left.Q\right|_{S}(C)=\left.Q(CS^{-1})\right|_{S}\in\mathbb{Q}X^{w}.
\]
Thus $Q=0$ if $w>0$. Hence the case $w>0$ is proved. Assume that
$w=0$. Then $Q(C)=Q(C')$ for any $C,C'\in\Gamma\backslash{\rm SL}(2,\mathbb{Z})$
since

\[
P(C)=\left.P\right|_{g}(C)=P(Cg^{-1})=P(C')
\]
where $g$ is an element of ${\rm SL}(2,\mathbb{Z})$ satisfying $Cg^{-1}=C'$.
Let $\alpha\coloneqq Q(C)$ which does not depend on the choice of
$C$. Then
\begin{align*}
\alpha & =\frac{1}{[{\rm SL}(2,\mathbb{Z}):\Gamma]}\sum_{C\in\Gamma\backslash{\rm SL}(2,\mathbb{Z})}Q(C)\\
 & =\frac{1}{[{\rm SL}(2,\mathbb{Z}):\Gamma]}\sum_{C\in\Gamma\backslash{\rm SL}(2,\mathbb{Z})}\left.P\right|_{1+U+U^{2}}(C)\\
 & =\frac{3}{[{\rm SL}(2,\mathbb{Z}):\Gamma]}\sum_{C\in\Gamma\backslash{\rm SL}(2,\mathbb{Z})}P(C)\\
 & =\frac{3}{2[{\rm SL}(2,\mathbb{Z}):\Gamma]}\sum_{C\in\Gamma\backslash{\rm SL}(2,\mathbb{Z})}\left.P\right|_{1+S}(C)\\
 & =0.
\end{align*}
Thus $Q=0$, which completes the proof.
\end{proof}
\begin{prop}
\label{prop:fA}Let $f_{A}:A^{+}\to A/A^{-}$ be a map defined by
\[
f_{A}(P)=(\left.P\right|_{1+U-U^{2}S}\bmod A^{-}),
\]
$B=A^{\vee}$ the dual vector space of $A$, and $f_{A}^{*}:B^{+}\to B/B^{-}$
the dual homomorphism of $f$. Then 
\[
\ker f_{A}^{*}=\mathcal{W}^{+}(B).
\]
Furthermore, there is an bijection
\[
\iota:\mathcal{W}^{+}(A)\simeq\ker(f_{A})
\]
given by $\iota(P)=\left.P\right|_{U(1+\epsilon)}$, and thus $\dim_{\mathbb{Q}}\mathcal{W}^{+}(A)=\dim_{\mathbb{Q}}\mathcal{W}^{+}(B)$.
\end{prop}

\begin{proof}
Note that $P$ is in $\ker f_{A}^{*}$ if and only if $\left.P\right|_{(1+U^{2}-SU)(1+\epsilon)}=0$.
For $P\in\mathcal{W}^{+}(B)$, we have
\[
\left.P\right|_{(1+U^{2}-SU)(1+\epsilon)}=\left.P\right|_{(1+U^{2}+U)(1+\epsilon)}=0,
\]
and thus $\mathcal{W}^{+}(B)\subset\ker f^{*}$. On the other hand,
if $P\in\ker f^{*}$, then
\begin{equation}
0=\left.P\right|_{(1+U^{2}-SU)(1+\epsilon)}=\left.P\right|_{2+U^{2}-SU+SUS-U^{2}S}.\label{eq:eA1}
\end{equation}
By acting $S$ to (\ref{eq:eA1}), we also have
\begin{equation}
0=\left.P\right|_{2S+U^{2}S-SUS+SU-U^{2}}.\label{eq:eA2}
\end{equation}
By adding (\ref{eq:eA1}) and (\ref{eq:eA2}), we get
\begin{equation}
\left.P\right|_{1+S}=0.\label{eq:eA3}
\end{equation}
By (\ref{eq:eA1}) and (\ref{eq:eA2}), we have
\begin{equation}
0=\left.P\right|_{(1+U^{2}+U)(1-S)}.\label{eq:eA4}
\end{equation}
By (\ref{eq:eA1}) and (\ref{eq:eA2}), we have $P\in\mathcal{W}^{+}(B)$,
which completes the proof of
\begin{equation}
\ker f_{A}^{*}=\mathcal{W}^{+}(B).\label{eq:ker_fAstar=00003DWplusB}
\end{equation}
The map $\iota$ is well-defined since
\begin{align*}
\left.P\right|_{U(1+\epsilon)(1+U-U^{2}S)(1+\epsilon)} & =\left.P\right|_{U(1+\epsilon)(U-U^{2}S)(1+\epsilon)}+2\left.P\right|_{U(1+\epsilon)}\\
 & =\left.P\right|_{U(1+\epsilon)(1-\epsilon S)(U-U^{2}S)}+2\left.P\right|_{U(1+\epsilon)}\\
 & =\left.P\right|_{(U-U\epsilon S)(1+\epsilon)(U-U^{2}S)}+2\left.P\right|_{U(1+\epsilon)}\\
 & =-2\left.P\right|_{(U-U^{2}S)}+2\left.P\right|_{U(1+\epsilon)}\\
 & =0
\end{align*}
for $P\in\mathcal{W}^{+}(A)$. If $P\in\mathcal{W}^{+}(A)$ and $\iota(P)=0$,
then
\[
0=\left.\iota(P)\right|_{1-S}=\left.P\right|_{U(1+\epsilon)(1-S)}=\left.P\right|_{(U-U^{2}S)(1-S)}=\left.P\right|_{(1+U+U^{2})(1-S)-(1-S)}=-2P.
\]
Thus $\iota$ is injective and
\begin{equation}
\dim_{\mathbb{Q}}\mathcal{W}^{+}(A)\leq\dim_{\mathbb{Q}}\ker(f_{A}).\label{eq:dim_ineq1}
\end{equation}
Since $\dim_{\mathbb{Q}}A^{+}=\dim_{\mathbb{Q}}(A/A^{-})$, we have
$\dim_{\mathbb{Q}}\ker(f_{A})=\dim_{\mathbb{Q}}\ker f_{A}^{*}$. Thus,
by (\ref{eq:ker_fAstar=00003DWplusB}), we have
\begin{equation}
\dim_{\mathbb{Q}}\ker(f_{A})=\dim_{\mathbb{Q}}\mathcal{W}^{+}(B).\label{eq:dim_eq1}
\end{equation}
By changing $A$ and $B$ in (\ref{eq:dim_ineq1}) and (\ref{eq:dim_eq1}),
we also have
\begin{equation}
\dim_{\mathbb{Q}}\mathcal{W}^{+}(B)\leq\dim_{\mathbb{Q}}\ker(f_{B})\label{eq:dim_ineq2}
\end{equation}
and
\begin{equation}
\dim_{\mathbb{Q}}\ker(f_{B})=\dim_{\mathbb{Q}}\mathcal{W}^{+}(A).\label{eq:dim_eq2}
\end{equation}
By (\ref{eq:dim_ineq1}), (\ref{eq:dim_eq1}), (\ref{eq:dim_ineq2}),
and (\ref{eq:dim_eq2}), we have
\[
\dim_{\mathbb{Q}}\mathcal{W}^{+}(A)=\dim_{\mathbb{Q}}\ker(f_{A})=\dim_{\mathbb{Q}}\mathcal{W}^{+}(B)=\dim_{\mathbb{Q}}\ker(f_{B}).
\]
Thus $\iota$ is bijection and the proposition is proved.
\end{proof}

\subsection{Proof of the main theorem}

Fix integers $N\geq1$ and $w\geq0$. For $C=C_{a,b}\in\Gamma_{1}(N)\backslash{\rm SL}(2,\mathbb{Z})$,
define $\mathbb{Q}$-linear maps $\lambda_{C}$, $\lambda_{C}^{\mathcal{S}}$
and $\lambda_{C}^{\mathcal{P}}$ from $V_{w}$ to $\mathcal{D}_{w+2,N}$
by

\begin{align*}
\lambda_{C}(X^{r}Y^{s}) & \coloneqq r!s!Z_{a,b}^{r+1,s+1},\\
\lambda_{C}^{\mathcal{S}}(X^{r}Y^{s}) & \coloneqq r!s!Z_{a+b}^{r+s+2},\\
\lambda_{C}^{\mathcal{P}}(X^{r}Y^{s}) & \coloneqq r!s!P_{a,b}^{r+1,s+1}.
\end{align*}
Furthermore, for $\bullet\in\{\emptyset,\mathcal{S},\mathcal{P}\}$,
define $\lambda^{\bullet}:\tilde{\bar{V}}_{w}^{\Gamma_{1}(N)}\to\mathcal{D}_{w+2,N}$
by
\[
\lambda^{\bullet}(P)=\sum_{C\in\Gamma_{1}(N)\backslash{\rm SL}(2,\mathbb{Z})}\lambda_{C}^{\bullet}(P(CS)).
\]
Let us rewrite the defining relations of $\mathcal{D}_{w+2,N}$:
\begin{equation}
P_{a,b}^{r+1,s+1}=Z_{a,b}^{r+1,s+1}+Z_{b,a}^{s+1,r+1}+Z_{a+b}^{r+s+2}=\sum_{j=0}^{r}\binom{s+j}{j}Z_{a-b,b}^{r-j+1,s+j+1}+\sum_{j=0}^{s-1}{r-1+j \choose j}Z_{b-a,a}^{s-j+,r+j+1}\label{eq:dsh1}
\end{equation}
by using $\lambda$, $\lambda^{\mathcal{S}}$ and $\lambda^{\mathcal{P}}$.
Put 
\[
A(N,w)=\{(r,s,a,b):r,s\geq0,\,r+s=w,\,(a,b)\in A(N)\}.
\]
For $(r,s,a,b)\in A(N,w)$, define $Q_{a,b}^{r,s}\in\tilde{\bar{V}}_{w}^{\Gamma_{1}(N)}$
by $Q_{a,b}^{r,s}(C)=\delta_{C_{a,b}S,C}\frac{X^{r}Y^{s}}{r!s!}$.
Note that for $m\in\{0,1\}$ and $\gamma\in\epsilon^{m}{\rm SL}(2,\mathbb{Z})$,
we have
\begin{align*}
\lambda^{\bullet}(\left.Q_{a,b}^{r,s}\right|_{\gamma}) & =\sum_{C}\lambda_{C}^{\bullet}(\left.Q_{a,b}^{r,s}\right|_{\gamma}(CS))=\sum_{C}\lambda_{C}^{\bullet}(\left.Q_{a,b}^{r,s}(J^{m}\epsilon^{m}CS\gamma^{-1})\right|_{\gamma})=\sum_{C}\delta_{C_{a,b}S,J^{m}\epsilon^{m}CS\gamma^{-1}}\lambda_{C}^{\bullet}(\left.X^{r}Y^{s}\right|_{\gamma})\\
 & =\lambda_{J^{m}\epsilon^{m}C_{a,b}S\gamma S^{-1}}^{\bullet}(\left.X^{r}Y^{s}\right|_{\gamma})=\lambda_{\epsilon^{m}C_{a,b}(\gamma^{t})^{-1}}(\left.X^{r}Y^{s}\right|_{\gamma})
\end{align*}
where $\gamma^{t}$ is the transposed matrix of $\gamma$. Then the
terms in (\ref{eq:dsh1}) can be described as follows:
\begin{align*}
P_{a,b}^{r+1,s+1} & =\frac{1}{r!s!}\lambda^{\mathcal{P}}(Q_{a,b}^{r,s}),\\
Z_{a,b}^{r+1,s+1} & =\frac{1}{r!s!}\lambda_{C_{a,b}}(X^{r}Y^{s})=\lambda(Q_{a,b}^{r,s}),\\
Z_{b,a}^{s+1,r+1} & =\frac{1}{r!s!}\lambda_{\epsilon C_{a,b}\epsilon S}(\left.X^{r}Y^{s}\right|_{\epsilon S})=\lambda(\left.Q_{a,b}^{r,s}\right|_{\epsilon S}),\\
Z_{a+b}^{r+s+2} & =\frac{1}{r!s!}\lambda^{\mathcal{S}}(Q_{a,b}^{r,s}),\\
\sum_{j=0}^{r}\binom{s+j}{j}Z_{a-b,b}^{r-j+1,s+j+1} & =\frac{1}{r!s!}\lambda_{C_{a-b,b}}(\left.X^{r}Y^{s}\right|_{T})=\lambda(\left.Q_{a,b}^{r,s}\right|_{T}),\\
\sum_{j=0}^{s}{r+j \choose j}Z_{b-a,a}^{s-j+1,r+j+1} & =\frac{1}{r!s!}\lambda_{C_{b-a,a}}(\left.X^{r}Y^{s}\right|_{\epsilon ST})=\lambda(\left.Q_{a,b}^{r,s}\right|_{\epsilon ST}).
\end{align*}
Thus (\ref{eq:dsh1}) is equivalent to
\begin{equation}
\lambda^{\mathcal{P}}(P)=\lambda(\left.P\right|_{(1+\epsilon S)})+\lambda^{\mathcal{S}}(P)=\lambda(\left.P\right|_{(1+\epsilon S)T})\label{eq:dsh2}
\end{equation}
for $P=Q_{a,b}^{r,s}\in\tilde{\bar{V}}_{w}^{\Gamma_{1}(N)}$. Let
$\bar{V}_{w}^{\Gamma,\pm}$ be the $\pm$-eigenspaces of $\bar{V}_{w}^{\Gamma}$
for the action of $\epsilon$.

Let $\delta:\bar{V}_{w}^{\Gamma}\to\bar{V}_{w}^{\Gamma}/\bar{V}_{w}^{\Gamma,-}$
be a map defined by
\[
\delta(P)=(\left.P\right|_{1+SU^{2}S-SU}\bmod\bar{V}_{w}^{\Gamma,-}).
\]
Then the induced dual map $\delta^{*}:V_{w}^{\Gamma,+}\to V_{w}^{\Gamma}$
is expressed as
\[
\delta^{*}(P)=\left.P\right|_{1+(SU^{2}S)^{-1}-(SU)^{-1}}=\left.P\right|_{1+SUS-U^{2}S}.
\]

\begin{prop}
\label{prop:kernel_delta_star}For an element $P$ of $V_{w}^{\Gamma,+}$,
the following conditions are equivalent.
\begin{enumerate}
\item $\delta^{*}(P)\in V_{w}^{\Gamma,-}$,
\item $\delta^{*}(P)=0$,
\item $P\in W_{w}^{\Gamma,+}$.
\end{enumerate}
\end{prop}

\begin{proof}
The implication (3)$\Rightarrow$(2) follows from the direct calculation
and the implication (2)$\Rightarrow$(1) is trivial. Thus the remaining
is a proof of (1)$\Rightarrow$(3). Assume that $\delta^{*}(P)\in V_{w}^{\Gamma,-}$
and $P\in V_{w}^{\Gamma,+}$. Then
\begin{align}
0 & =\left.\delta^{*}(P)\right|_{1+\epsilon}=\left.P\right|_{(1+SUS-U^{2}S)(1+\epsilon)}=\left.P\right|_{2+SUS-U^{2}S+U^{2}-SU}.\label{eq:e1}
\end{align}
By acting $S$ to (\ref{eq:e1}), we also have
\begin{equation}
0=\left.P\right|_{2S+SU-U^{2}+U^{2}S-SUS}.\label{eq:e2}
\end{equation}
By adding (\ref{eq:e1}) and (\ref{eq:e2}), we obtain
\begin{equation}
\left.P\right|_{1+S}=0.\label{eq:e3}
\end{equation}
By (\ref{eq:e2}) and (\ref{eq:e3}), we have
\begin{equation}
\left.P\right|_{(1+U+U^{2})(S-1)}=0.\label{eq:e4}
\end{equation}
By (\ref{eq:e3}), (\ref{eq:e4}) and Proposition \ref{prop:W_general},
we have $P\in W_{w}^{\Gamma,+}$. Hence (1)$\Rightarrow$(3) is also
proved.
\end{proof}
\begin{lem}
\label{lem:dsh_in_kerdelta}We have
\[
\lambda^{-1}(\mathcal{P}_{k,N}^{{\rm ev}})\subset\ker\delta.
\]
\end{lem}

\begin{proof}
By definition, $\lambda^{-1}(\mathcal{P}_{k,N}^{{\rm ev}})$ is spanned
by
\[
\{\left.P\right|_{1+\epsilon S}:P\in\bar{V}_{w}^{\Gamma_{1}(N),-}\}\cup\{\left.P\right|_{(1+\epsilon S)(1-T)}:P\in\bar{V}_{w}^{\Gamma_{1}(N)}\}.
\]
Note that $Q\in\ker\delta$ if and only if $\left.Q\right|_{(1+SU^{2}S-SU)(1+\epsilon)}=0$.
For $P\in\bar{V}_{w}^{\Gamma_{1}(N)}$, we have
\[
\left.P\right|_{(1+\epsilon S)(1+SU^{2}S-SU)(1+\epsilon)}=\left.P\right|_{(1+\epsilon S)(1+SU^{2}S-SU)(1+\epsilon)}=\left.P\right|_{(1+\epsilon)(1+S)},
\]
where the last expression vanishes if $P\in\bar{V}_{w}^{\Gamma_{1}(N),-}$.
Thus 
\[
\{\left.P\right|_{1+\epsilon S}:P\in\bar{V}_{w}^{\Gamma_{1}(N),-}\}\subset\ker\delta.
\]
Furthermore, for $P\in\bar{V}_{w}^{\Gamma_{1}(N)}$, we have
\[
\left.P\right|_{(1+\epsilon S)(1-T)(1+SU^{2}S-SU)(1+\epsilon)}=\left.P\right|_{(1+\epsilon S)(1-\epsilon S)(1+SU^{2}S-SU-US-S+U^{2})}=0.
\]
Thus, we also have
\[
\{\left.P\right|_{(1+\epsilon S)(1-T)}:P\in\bar{V}_{w}^{\Gamma_{1}(N)}\}\subset\ker\delta.
\]
Hence the lemma is proved.
\end{proof}
\begin{prop}
\label{prop:P_R_injectivity}Let $P\in W_{w}^{\Gamma,+}$. If $\left.P\right|_{U(1+\epsilon)}=0$
then $P=0$.
\end{prop}

\begin{proof}
It is just a special case of injectivity of $\iota$ in Proposition
\ref{prop:fA}.
\end{proof}
Note that 
\[
\mathcal{P}_{k,N}^{{\rm ev}}=\lambda(V_{w}^{\Gamma_{1}(N),-,{\rm sym}})+\lambda^{\mathcal{S}}(V_{w}^{\Gamma_{1}(N)})
\]
where
\[
V_{w}^{\Gamma_{1}(N),-,{\rm sym}}\coloneqq\{u\in V_{w}^{\Gamma_{1}(N),-}:\left.u\right|_{\epsilon S}=u\}.
\]
The following is a refined version of Theorem \ref{thm:main_intro}.
\begin{thm}
\label{thm:main}Let $P\in\bar{W}_{w}^{\Gamma,+}$. Put $Q\coloneqq\left.P\right|_{U}$
and $Q^{\pm}=\frac{1}{2}\left.Q\right|_{1\pm\epsilon}\in V_{w}^{\Gamma,\pm}$.
Then 
\[
Q^{-}\in V_{w}^{\Gamma_{1}(N),-,{\rm sym}}
\]
and
\[
3\lambda(Q^{+})=-\lambda(Q^{-})-\lambda^{\mathcal{S}}(Q)\in\mathcal{P}_{k,N}^{{\rm ev}}.
\]
Conversely, $\lambda(R)$ with \textbf{$R\in\bar{V}_{w}^{\Gamma,+}$}
belongs to $\mathcal{P}_{k,N}^{{\rm ev}}$ if and only if $R=Q^{+}$
arising in this way.
\end{thm}

\begin{proof}
Note that
\[
\left.Q\right|_{\epsilon S-1}=\left.P\right|_{U(\epsilon S-1)}=P,
\]
and thus
\[
\left.Q^{-}\right|_{\epsilon S-1}=\frac{1}{2}\left.Q\right|_{(\epsilon S-1)(1-\epsilon)}=\frac{1}{2}\left.P\right|_{1-\epsilon}=0,
\]
which implies

\[
Q^{-}\in V_{w}^{\Gamma_{1}(N),-,{\rm sym}}.
\]
By (\ref{eq:dsh2}),
\begin{align*}
-\lambda^{\mathcal{S}}(Q) & =\lambda(\left.Q\right|_{(1+\epsilon S)(1-T)})\\
 & =\lambda(\left.P\right|_{U(1+\epsilon S)(1-T)})\\
 & =\lambda(\left.P\right|_{U(2+\epsilon)})\\
 & =\frac{1}{2}\lambda(\left.P\right|_{U(2+\epsilon)(1+\epsilon)})+\frac{1}{2}\lambda(\left.P\right|_{U(2+\epsilon)(1-\epsilon)}).
\end{align*}
Here
\[
\frac{1}{2}\lambda(\left.P\right|_{U(2+\epsilon)(1+\epsilon)})=\frac{3}{2}\lambda(\left.P\right|_{U(1+\epsilon)})=3\lambda(Q^{+})
\]
and
\begin{align*}
\frac{1}{2}\lambda(\left.P\right|_{U(2+\epsilon)(1-\epsilon)}) & =\frac{1}{2}\lambda(\left.P\right|_{U(1-\epsilon)})=\lambda(Q^{-}).
\end{align*}
Thus
\[
3\lambda(Q^{+})=-\lambda(Q^{-})-\lambda^{\mathcal{S}}(Q),
\]
which completes the proof Theorem \ref{thm:main} except for the converse
part.

By Lemma \ref{lem:dsh_in_kerdelta},
\begin{equation}
\dim_{\mathbb{Q}}\left(\bar{V}_{w}^{\Gamma,+}\cap\lambda^{-1}(\mathcal{P}_{k,N}^{{\rm ev}})\right)\leq\dim_{\mathbb{Q}}(\bar{V}_{w}^{\Gamma,+}\cap\ker\delta).\label{eq:ineq1}
\end{equation}
By Proposition \ref{prop:kernel_delta_star}, 
\begin{equation}
\dim_{\mathbb{Q}}\left(\bar{V}_{w}^{\Gamma,+}\cap\ker(\delta)\right)=\dim_{\mathbb{Q}}W_{w}^{\Gamma,+}.\label{eq:eq2}
\end{equation}
By Proposition \ref{prop:fA}, we have
\begin{equation}
\dim_{\mathbb{Q}}W_{w}^{\Gamma,+}=\dim_{\mathbb{Q}}\bar{W}_{w}^{\Gamma,+}.\label{eq:eq3}
\end{equation}
By (\ref{eq:ineq1}), (\ref{eq:eq2}), and (\ref{eq:eq3}), we have
\begin{equation}
\dim_{\mathbb{Q}}\left(\bar{V}_{w}^{\Gamma,+}\cap\lambda^{-1}(\mathcal{P}_{k,N}^{{\rm ev}})\right)\leq\dim_{\mathbb{Q}}\bar{W}_{w}^{\Gamma,+}.\label{eq:tar_ineq_1}
\end{equation}
Therefore, the converse part follows from (\ref{eq:tar_ineq_1}) since
the $\mathbb{Q}$-linear map $P\mapsto Q^{+}$ embeds $\bar{W}_{w}^{\Gamma,+}$
into $\bar{V}_{w}^{\Gamma,+}\cap\lambda^{-1}(\mathcal{P}_{k,N}^{{\rm ev}})$
by Proposition \ref{prop:P_R_injectivity}.
\end{proof}
If we explicitly write down Theorem \ref{thm:main} by the coefficients
of period polynomials, we can show Theorem \ref{thm:main_intro} as
follows.
\begin{proof}[Proof of Theorem \ref{thm:main_intro}]
By definition,
\[
P(C_{a,-a+b})(X+Y,-X)=\left.P\right|_{U}(C_{a,b}S),
\]
and thus
\[
\left.P\right|_{U}(C_{a,b}S)=\sum_{r,s}\frac{1}{r!s!}q_{a,b}^{r+1,s+1}X^{r}Y^{s}.
\]
Thus
\[
\left.P\right|_{U}=\sum_{(r,s,a,b)\in A(w,N)}q_{a,b}^{r+1,s+1}Q_{a,b}^{r,s}.
\]
Let $Q=\left.P\right|_{U}$ and $Q^{\pm}=\frac{1}{2}\left.Q\right|_{1\pm\epsilon}$
as in Theorem \ref{thm:main}. Then
\begin{align*}
Q & =\sum_{(r,s,a,b)\in A(w,N)}q_{a,b}^{r+1,s+1}Q_{a,b}^{r,s},\\
Q^{+} & =\sum_{(r,s,a,b)\in A(w,N)}q_{a,b}^{r+1,s+1,{\rm od}}Q_{a,b}^{r,s},\\
Q^{-} & =\sum_{(r,s,a,b)\in A(w,N)}q_{a,b}^{r+1,s+1,{\rm ev}}Q_{a,b}^{r,s},
\end{align*}
and thus the conditions
\[
q_{b,a}^{s,r,{\rm od}}=(-1)^{r+1}q_{-a,b}^{r,s,{\rm od}}=(-1)^{s+1}q_{a,-b}^{r,s,{\rm od}},\quad q_{a,b}^{r,s,{\rm ev}}=(-1)^{r}q_{-a,b}^{r,s,{\rm ev}}=(-1)^{s}q_{a,-b}^{r,s,{\rm ev}}=q_{b,a}^{s,r,{\rm ev}}
\]
follows from $Q^{\pm}\in\tilde{\bar{V}}_{w}^{\Gamma,\pm}$ and $Q^{-}\in V_{w}^{\Gamma_{1}(N),-,{\rm sym}}$.
Furthermore, by
\[
3\lambda(Q^{+})=-\lambda(Q^{-})-\lambda^{\mathcal{S}}(Q),
\]
we have
\[
3\sum_{r,s,a,b}q_{a,b}^{r,s,{\rm od}}Z_{a,b}^{r,s}=-\sum_{r,s,a,b}q_{a,b}^{r,s,{\rm ev}}Z_{a,b}^{r,s}-\sum_{r,s,a,b}q_{a,b}^{r,s}Z_{a+b}^{r+s}
\]
since $\lambda(Q_{a,b}^{r,s})=Z_{a.b}^{r+1,s+1}$ and $\lambda^{\mathcal{S}}(Q_{a,b}^{r,s})=Z_{a,b}^{r+s+2}$.
Finally, the converse part follows from the converse part of Theorem
\ref{thm:main} since
\[
\sum_{r,s,a,b}c_{a,b}^{r,s}Z_{a,b}^{r,s}=\lambda(\sum_{r,s,a,b}c_{a,b}^{r,s}Q_{a,b}^{r-1,s-1})
\]
and
\[
\sum_{r,s,a,b}c_{a,b}^{r,s}Q_{a,b}^{r-1,s-1}\in\bar{V}_{w}^{\Gamma,+}.\qedhere
\]
\end{proof}

\section{Examples}

\subsection{The case where the weight is $2$, the level is an odd prime $p$,
and $\Gamma=\Gamma_{0}(p)$ }

Let $p$ be an odd prime. By definition, $\bar{W}_{0}^{\Gamma_{0}(p),+}$
is a set of maps $P$ from $\Gamma_{0}(p)\backslash{\rm SL}(2,\mathbb{Z})$
to $V_{0}=\mathbb{Q}$ satisfying
\[
\sum_{\gamma\in\{1,S\}}P(C\gamma^{-1})=\sum_{\gamma\in\{1,U,U^{2}\}}P(C\gamma^{-1})=0
\]
and
\[
P(C)=P(\epsilon C\epsilon)
\]
for $C\in\Gamma_{0}(p)\backslash{\rm SL}(2,\mathbb{Z})$. Furthermore,
$\bar{C}_{0}^{\Gamma_{0}(p),+}=\mathbb{Q}P^{\mathrm{Eis}}$ where
\[
P^{\mathrm{Eis}}(C_{\alpha}):=\begin{cases}
1 & \alpha=(1:0)\\
-1 & \alpha=(0:1)\\
0 & \alpha\neq(1:0),(0:1).
\end{cases}
\]
Note that
\[
\bar{W}_{0}^{\Gamma_{0}(p),+}=\mathbb{Q}P^{\mathrm{Eis}}\oplus\left\{ P\in\bar{W}_{0}^{\Gamma_{0}(p),+}\,:\,P(C_{(1:0)})=P(C_{(0:1)})=0\right\} .
\]

\begin{example}
Let us consider Theorem \ref{thm:main_intro_gamma0} and Corollary
\ref{cor:intro_main_gamma0} for the case where $P=P^{\mathrm{Eis}}$.
Then
\[
q_{(a:b)}^{1,1}=P^{\mathrm{Eis}}(C_{(a:-a+b)})=\begin{cases}
1 & (a:b)=(1:1)\\
-1 & (a:b)=(0:1)\\
0 & (a:b)\neq(1:1),(0:1).
\end{cases}
\]
Thus, Theorem \ref{thm:main_intro_gamma0} says that
\begin{equation}
3\left(\frac{1}{2}(Z_{(1:1)}^{1,1}+Z_{(-1:1)}^{1,1})-Z_{(0:1)}^{1,1}\right)=\frac{1}{2}\left(Z_{(1:1)}^{1,1}-Z_{(-1:1)}^{1,1}\right)+Z_{(0:1)}^{2}-Z_{(-1:1)}^{2}.\label{eq:exa_gamma0_Z_1}
\end{equation}
By applying $\Phi_{2,p}$ to \ref{eq:exa_gamma0_Z_1}, we get
\[
L_{p}^{+}(1)-L_{p}^{+}(0)=\frac{2}{3}\left(\frac{1}{2}L_{p}^{-}(1)+\sum_{c=1}^{p-1}\zeta\binom{2}{\zeta_{N}^{c}}-(p-1)\zeta(2)\right).
\]
and Corollary \ref{cor:intro_main_gamma0} says that
\[
L_{p}^{+}(1)-L_{p}^{+}(0)\in\mathbb{Q}\pi^{2}
\]
(see Section \ref{subsec:intro_X0_11} for the definition of $L_{p}^{+}(a)$).
\end{example}

\begin{example}
Let us consider Theorem \ref{thm:main_intro_gamma0} and Corollary
\ref{cor:intro_main_gamma0} for the case where $P\in\bar{W}_{0}^{\Gamma_{0}(p),+}$
and $P(C_{(1:0)})=P(C_{(0:1)})=0$. Since $P\in\bar{W}_{0}^{\Gamma_{0}(p),+}$,
$P(C_{(1:1)})=P(C_{(-1:1)})=0$. Thus, $q_{(1:0)}^{1,1}=P(C_{(-1:1)})=0$
and
\[
q_{(a:1)}^{1,1}=P(C_{(a:1-a)}).
\]
Thus, Theorem \ref{thm:main_intro_gamma0} says that
\[
3\sum_{a=1}^{p-1}(P(C_{(a:1-a)})+P(C_{(a:1+a)}))Z_{(a:1)}^{1,1}=\sum_{a=1}^{p-1}(P(C_{(a:1-a)})-P(C_{(a:1+a)}))Z_{(a:1)}^{1,1}-\sum_{a=1}^{p-1}P(C_{(a:1-a)})Z_{(a:1)}^{2},
\]
and Corollary \ref{cor:intro_main_gamma0} says that
\begin{equation}
\sum_{a=1}^{(p-1)/2}\left(P(C_{(a:1-a)})+P(C_{(a:1+a)})\right)L_{p}^{+}(a)\in\mathbb{Q}\pi^{2}.\label{eq:Lp_in_Qpi2}
\end{equation}
Conjecturally, for each odd prime $p$, (\ref{eq:Lp_in_Qpi2}) gives
all $\mathbb{Q}$-linear relations among $\pi^{2}$ and $L_{p}^{+}(a)$
for $a=1,\dots,(p-1)/2$.
\end{example}

\subsection{The case where the level is $2$}

Let $\Gamma=\Gamma_{1}(2)=\Gamma_{0}(2)$. Put $k=w+2\geq0$. Then,
for $k\geq2,$
\[
\dim_{\mathbb{Q}}M_{k}(\Gamma)=\dim_{\mathbb{Q}}\bar{W}_{w}^{\Gamma_{1}(N),+}=\begin{cases}
1+\left[\frac{k}{4}\right] & k:\mathrm{even}\\
0 & k:\mathrm{odd}
\end{cases}
\]
and
\[
\dim_{\mathbb{Q}}S_{k}(\Gamma)=\dim_{\mathbb{Q}}\left(\bar{W}_{w}^{\Gamma_{1}(N),+}/\bar{C}_{w}^{\Gamma,\pm}\right)=\begin{cases}
\left[\frac{k}{4}\right]-1+\delta_{k,2} & k:\mathrm{even}\\
0 & k:\mathrm{odd}
\end{cases}
\]
By definition, for even $w$, $\bar{W}_{w}^{\Gamma_{1}(N),+}$ can
be regarded as the set of tuple $(p_{0,1},p_{1,0},p_{1,1})\in V_{w}^{3}$
satisfying
\[
p_{0,1}(X,Y)+p_{1,0}(-Y,X)=0,
\]
\[
p_{1,1}(X,Y)+p_{1,1}(Y,-X)=0,
\]
and
\[
p_{0,1}(X,Y)+p_{1,1}(X-Y,X)+p_{1,0}(-Y,X-Y)=0.
\]
Furthermore, $\bar{C}_{w}^{\Gamma,\pm}$ is generated by
\[
(Y^{w},-X^{w},0)
\]
and
\[
(X^{w},-Y^{w},X^{w}-Y^{w}).
\]

\begin{example}
Let us consider Theorem \ref{thm:main_intro} for the case $P=(Y^{w},-X^{w},0)$.
Then, for $(a,b)=(0,1)$,
\[
P(C_{a,-a+b})(X-Y,X)=X^{w},
\]
and thus
\[
q_{0,1}^{r+1,s+1}=\delta_{r,w}w!.
\]
For $(a,b)=(1,0),$
\[
P(C_{a,-a+b})(X-Y,X)=0,
\]
and thus
\[
q_{1,0}^{r+1,s+1}=0.
\]
For $(a,b)=(1,1),$
\[
P(C_{a,-a+b})(X-Y,X)=-(X-Y)^{w},
\]
and thus
\[
q_{1,1}^{r+1,s+1}=(-1)^{r+1}w!.
\]
Thus
\[
q_{0,1}^{r+1,s+1,\mathrm{od}}=\delta_{r,w}w!,\qquad q_{1,0}^{r+1,s+1,\mathrm{od}}=0,\qquad q_{1,1}^{r+1,s+1,\mathrm{od}}=\begin{cases}
-w! & r:\mathrm{even}\\
0 & r:\mathrm{odd},
\end{cases}
\]
\[
q_{0,1}^{r+1,s+1,\mathrm{ev}}=0,\qquad q_{1,0}^{r+1,s+1,\mathrm{ev}}=0,\qquad q_{1,1}^{r+1,s+1,\mathrm{ev}}=\begin{cases}
0 & r:\mathrm{even}\\
w! & r:\mathrm{odd},
\end{cases}
\]
and Theorem \ref{thm:main_intro} says that
\[
Z_{0,1}^{k-1,1}-\sum_{\substack{m+n=k\\
m,n:\mathrm{odd}
}
}Z_{1,1}^{m,n}=-\frac{1}{3}\sum_{\substack{m+n=k\\
m,n:\mathrm{even}
}
}Z_{1,1}^{m,n}-\frac{1}{3}Z_{1}^{k}+\frac{1}{3}Z_{0}^{k}.
\]
Corollary \ref{cor:intro_main} says that
\[
\zeta\binom{k-1,1}{1,-1}-\sum_{\substack{m+n=k\\
m,n:\mathrm{odd}
}
}\zeta\binom{m,n}{-1,-1}\in\mathbb{Q}\cdot(2\pi i)^{k}.
\]
\end{example}

\begin{example}
Let us consider Theorem \ref{thm:main_intro} for the case $P=(X^{w},-Y^{w},X^{w}-Y^{w})$.
Then, for $(a,b)=(0,1)$,
\[
P(C_{a,-a+b})(X-Y,X)=(X-Y)^{w},
\]
and thus
\[
q_{0,1}^{r+1,s+1}=(-1)^{r}w!.
\]
For $(a,b)=(1,0),$
\[
P(C_{a,-a+b})(X-Y,X)=(X-Y)^{w}-X^{w},
\]
and thus
\[
q_{1,0}^{r+1,s+1}=(-1)^{r}(1-\delta_{r,w})w!.
\]
For $(a,b)=(1,1),$
\[
P(C_{a,-a+b})(X-Y,X)=-X^{w},
\]
and thus
\[
q_{1,1}^{r+1,s+1}=-\delta_{r,w}w!.
\]
Thus
\[
q_{0,1}^{r+1,s+1,\mathrm{od}}=\begin{cases}
w! & r:\mathrm{even}\\
0 & r:\mathrm{odd}
\end{cases},\qquad q_{1,0}^{r+1,s+1,\mathrm{od}}=\begin{cases}
(1-\delta_{r,w})w! & r:\mathrm{even}\\
0 & r:\mathrm{odd}
\end{cases},\qquad q_{1,1}^{r+1,s+1,\mathrm{od}}=-\delta_{r,w}w!
\]
\[
q_{0,1}^{r+1,s+1,\mathrm{ev}}=\begin{cases}
0 & r:\mathrm{even}\\
-w! & r:\mathrm{odd}
\end{cases},\qquad q_{1,0}^{r+1,s+1,\mathrm{ev}}=\begin{cases}
0 & r:\mathrm{even}\\
-w! & r:\mathrm{odd}
\end{cases},\qquad q_{1,1}^{r+1,s+1,\mathrm{ev}}=0
\]
and Theorem \ref{thm:main_intro} and Corollary \ref{cor:intro_main}
say that
\[
\sum_{\substack{m+n=k\\
m,n:\mathrm{odd}
}
}Z_{0,1}^{m,n}+\sum_{\substack{m+n=k\\
m,n:\mathrm{odd}\\
1\leq m\leq k-3
}
}Z_{1,0}^{m,n}-Z_{1,1}^{k-1,1}=\frac{1}{3}\sum_{\substack{m+n=k\\
m,n:\mathrm{even}
}
}Z_{0,1}^{m,n}+\frac{1}{3}\sum_{\substack{m+n=k\\
m,n:\mathrm{even}
}
}Z_{1,0}^{m,n}-\frac{1}{3}Z_{1}^{k}+\frac{1}{3}Z_{1}^{k},
\]
and
\[
\sum_{\substack{m+n=k\\
m,n:\mathrm{odd}
}
}\zeta{m,n \choose 1,-1}+\sum_{\substack{m+n=k\\
m,n:\mathrm{odd}\\
1\leq m\leq k-3
}
}\zeta{m,n \choose -1,1}-\zeta{k-1,1 \choose -1,-1}\in\mathbb{Q}\cdot(2\pi i)^{k}.
\]
\end{example}

\begin{example}
The first case where a period polynomial not in $\bar{C}_{w}^{\Gamma,\pm}$
is the case $k=8$. Then
\[
P=(X^{2}Y^{4}-2X^{4}Y^{2},-X^{4}Y^{2}+2X^{2}Y^{4},(Y^{2}-X^{2})^{3})
\]
is in
\[
\bar{W}_{w}^{\Gamma_{1}(N),+}.
\]
Then, for $(a,b)=(0,1)$, 
\begin{align*}
P(C_{a,-a+b})(X-Y,X) & =(X-Y)^{2}X^{4}-2(X-Y)^{4}X^{2}\\
 & =-X^{6}+6X^{5}Y-11X^{4}Y^{2}+8X^{3}Y^{3}-2X^{2}Y^{4}\\
 & =-720\frac{X^{6}}{6!0!}+720\frac{X^{5}Y}{5!1!}-528\frac{X^{4}Y^{2}}{4!2!}+288\frac{X^{3}Y^{3}}{3!3!}-96\frac{X^{2}Y^{4}}{2!4!},
\end{align*}
and thus
\[
(q_{0,1}^{7,1},\dots,q_{0,1}^{1,7})=(-720,720,-528,288,-96,0,0).
\]
For $(a,b)=(1,0)$,
\begin{align*}
P(C_{a,-a+b},X-Y,X) & =(2XY-Y^{2})^{3}\\
 & =8X^{3}Y^{3}-12X^{2}Y^{4}+6XY^{5}-Y^{6}\\
 & =288\frac{X^{3}Y^{3}}{3!3!}--576\frac{X^{2}Y^{4}}{2!4!}+720\frac{XY^{5}}{1!5!}-720\frac{Y^{6}}{0!6!}
\end{align*}
and thus
\[
(q_{1,0}^{7,1},\dots,q_{1,0}^{1,7})=(0,0,0,288,-576,720,-720).
\]
For, $(a,b)=(1,1)$,
\begin{align*}
P(C_{a,-a+b},X-Y,X) & =-(X-Y)^{4}X^{2}+2(X-Y)^{2}X^{4}\\
 & =X^{6}-4X^{4}Y^{2}+4X^{3}Y^{3}-X^{2}Y^{4}\\
 & =720\frac{X^{6}}{6!}-192\frac{X^{4}Y^{2}}{4!2!}+144\frac{X^{3}Y^{3}}{3!3!}-48\frac{X^{2}Y^{4}}{2!4!},
\end{align*}
and thus
\[
(q_{1,1}^{7,1},\dots,q_{1,1}^{1,7})=(720,0,-192,144,-48,0,0).
\]
Note that
\[
\sum_{r,s}q_{0,1}^{r,s}+\sum q_{1,0}^{r,s}=-624,\quad\sum_{r,s}q_{1,1}^{r,s}=624.
\]
Thus, Theorem \ref{thm:main_intro} and Corollary \ref{cor:intro_main}
say that
\begin{align*}
 & -720Z_{0,1}^{7,1}-528Z_{0,1}^{5,3}-96Z_{0,1}^{3,5}-576Z_{1,0}^{3,5}-720Z_{1,0}^{1,7}+720Z_{1,1}^{7,1}-192Z_{1,1}^{5,3}-48Z_{1,1}^{3,5}\\
 & =-\frac{1}{3}\left(720Z_{0,1}^{6,2}+288Z_{0,1}^{4,4}+288Z_{1,0}^{4,4}+720Z_{1,0}^{2,6}+144Z_{1,1}^{4,4}\right)-\frac{1}{3}\left(-624Z_{1}^{8}+624Z_{0}^{8}\right)
\end{align*}
and
\[
-720\zeta{7,1 \choose 0,1}-528\zeta{5,3 \choose 0,1}-96\zeta{3,5 \choose 0,1}-576\zeta{3,5 \choose 1,0}-720\zeta{1,7 \choose 1,0}+720\zeta{7,1 \choose 1,1}-192\zeta{5,3 \choose 1,1}-48\zeta{3,5 \choose 1,1}\in\mathbb{Q}\cdot(2\pi i)^{8}.
\]
\end{example}

\begin{example}
The converse of Corollary \ref{cor:intro_main} is not true in general.
For example,
\[
7\zeta{3,1 \choose 1,-1}+3\zeta{1,3 \choose -1,1}=\frac{11\pi^{4}}{720}\in\mathbb{Q}\cdot(2\pi i)^{4},
\]
but,
\[
7Z_{0,1}^{3,1}+3Z_{1,0}^{1,3}\notin\mathcal{P}_{4,2}^{{\rm ev}}.
\]
\end{example}

\subsection*{Acknowledgements}

The author thanks Katsumi Kina for helpful comments. This work was
supported by JSPS KAKENHI Grant Numbers JP18K13392 and JP22K03244.

\bibliographystyle{plain}
\bibliography{reference}

\end{document}